\documentclass[11pt]{amsart}
\usepackage[utf8]{inputenc}
\usepackage[T1]{fontenc}
\usepackage{amsmath}
\usepackage{amssymb}
\usepackage{eulerpx,bm}
\usepackage{comment}
\usepackage{longtable}
\usepackage{accents}
\usepackage{xcolor,soul} 
\usepackage{colortbl}
\usepackage{hyperref}
\usepackage{cleveref}
\newtheorem{thm}{Theorem}[section]
\newtheorem{prop}[thm]{Proposition}

\newtheorem{cor}[thm]{Corollary}

\newtheorem*{obs*}{Observation}

\theoremstyle{definition}
\newtheorem{defn}[thm]{Definition}

\newtheorem{exmp}[thm]{Example}

\newtheorem{innercustomqu}{Open Problem}
\newenvironment{customqu}[1]
{\renewcommand\theinnercustomqu{#1}\innercustomqu}
{\endinnercustomqu}

\theoremstyle{remark}

\numberwithin{equation}{section}

\newcommand{\R}{\mathbb{R}}  % The real numbers.
\newcommand{\C}{\mathbb{C}}  % The complex numbers.
\newcommand{\Z}{\mathbb{Z}}  % The integers.
\newcommand{\N}{\mathbb{N}}  % The natural numbers.
\newcommand{\PP}{\mathbb{P}}  % Projective space.
\counterwithout{equation}{section}

\newcommand\quotient[2]{
	\mathchoice
	{% \displaystyle
		\text{\raise1ex\hbox{$#1$}\Big/\lower1ex\hbox{$#2$}}%
	}
	{% \textstyle
		#1\,/\,#2
	}
	{% \scriptstyle
		#1\,/\,#2
	}
	{% \scriptscriptstyle  
		#1\,/\,#2
	}
}
\begin{document}

\title[Separating Cones: Some Properties \& Open Problems]{Separating Cones defined by Toric Varieties: Some Properties and Open Problems}
\author{Charu Goel}
\address{Department of Basic Sciences, Indian Institute of Information Technology, Nagpur, India}
\email{charugoel@iiitn.ac.in}
%\urladdr{www.math.sc.edu/$\sim$howard} % Delete if not wanted.
%%
%% If there is another author uncomment and edit the following.
%%
\author{Sarah Hess}
\address{Department of Mathematics and Statistics, University of Konstanz, Konstanz, Germany}
\email{Sarah.Hess@uni-konstanz.de}
%\urladdr{www.math.sc.edu/$\sim$second}
%%
%% If there are three of more authors they are added in the obvious
%% way. 
\author{Salma Kuhlmann}
\address{Department of Mathematics and Statistics, University of Konstanz, Konstanz, Germany}
\email{Salma.Kuhlmann@uni-konstanz.de}
%\footnote[1]{Draft 6, \today}
%%%
%%% The following is for the abstract.  The abstract is optional and
%%% if not used just delete, or comment out, the following.
%%%
\begin{abstract}
	%For $n,d\in\N$, the cone $\mathcal{P}_{n+1,2d}$ of positive semidefinite homogeneous polynomials with real coefficients in $n+1$ variables of degree $2d$ contains the cone $\Sigma_{n+1,2d}$ of those that are representable as finite sums of squares. 
	In 1888, Hilbert proved that the cone $\mathcal{P}_{n+1,2d}$ of positive semidefinite forms in $n+1$ variables of degree $2d$ coincides with its subcone $\Sigma_{n+1,2d}$ of those forms that are representable as finite sums of squares if and only if $(n+1,2d)=(2,2d)_{d\geq 1}$ or $(n+1,2)_{n\geq 1}$ or $(3,4)$. These are the \textit{Hilbert cases}. In \cite{GHK, GHK2}, we applied the Gram matrix method to construct cones between $\Sigma_{n+1,2d}$ and $\mathcal{P}_{n+1,2d}$, defined by projective varieties containing the Veronese variety. In particular, we introduced and examined a specific cone filtration $$\Sigma_{n+1,2d}=C_0\subseteq \ldots \subseteq C_n \subseteq C_{n+1} \subseteq \ldots \subseteq C_{k(n,d)-n}=\mathcal{P}_{n+1,2d}$$ and determined each strict inclusion in non-Hilbert cases. This gave us a refinement of Hilbert's 1888 theorem.  Here, $k(n,d)+1$ is the dimension of the vector space of forms in $n+1$ variables of degree $d$. %In a non-Hilbert case, at least one of the inclusions has to be strict, but it is not clear which one and how many. The identification of each strict inclusion is the goal of this thesis.

	In this paper, we show that the %strictly separating 
	intermediate cones $C_i$'s are closed and describe their interiors and boundaries. We discuss the membership problem for the $C_i$'s, present open problems concerning their dual cones %Membership tests are in particular crucial for application while the dual cones $C_i^\vee$'s are related to the moment problem from functional analysis. 
	and generalizations to cones defined by toric varieties. 

\end{abstract}
%%
%%  LaTeX will not make the title for the paper unless told to do so.
%%  This is done by uncommenting the following.
%%
\maketitle %\footnote{Draft 4, 24/10/24}
%%
%% LaTeX can automatically make a table of contents.  This is done by
%% uncommenting the following:
%%
%\tableofcontents
%%%%%%%%%%%%%%%%%%%%%%%%%%%%%%%%%%%%%%%%%%%%%%%%%%%%%%%%%%%%%%%%%%%%%%%%%%%%%%%%%%%%%%%%%%%%%%%%%%%%%%%%%%%%%%%%%%%%%%%%%%%%%%%%%%%%%
%%%%%%%%%%%%%%%%%%%%%%%%%%%%%%%%%%%%%%%%%%%%%%%%%%%%%%%%%%%%%%%%%%%%%%%%%%%%%%%%%%%%%%%%%%%%%%%%%%%%%%%%%%%%%%%%%%%%%%%%%%%%%%%%%%%%%
\section{Introduction}

%\subsection{Preliminaries}
%For $n,\, l\in\N$, we denote the vector space of 
A real form (i.e., a homogeneous polynomial) is called %in $n+1$ variables of degree $l$ by $\mathcal{F}_{n+1,l}$ and call $f\in\mathcal{F}_{n+1,l}$ a \textit{$(n+1)$-ary $l$-ic}. For $K\subseteq\R^{n+1}$, we call $f$
\textit{(globally) positive semi-definite (PSD)} if it only takes non-negative values and is called a \textit{sum of squares (SOS)} if it is a finite sum of squares of half degree real forms. Let $\mathcal{F}_{n+1,l}$ denote the vector space of real forms in $n+1$ variables of degree $l$, where $n,l\in\N$. Call $f\in\mathcal{F}_{n+1,l}$ a \textit{$(n+1)$-ary $l$-ic}. For $K\subseteq\R^{n+1}$, $f\in\mathcal{F}_{n+1,l}$ is called \textit{locally PSD} on $K\subseteq\R^{n+1}$, if $f(x)\geq 0$ for all $x\in K$. In the special case of $K=\R^{n+1}$, $f$ is (globally) PSD. 
For $d\in\N$, the set of all PSD forms in $\mathcal{F}_{n+1,2d}$, denoted by $\mathcal{P}_{n+1,2d}$, is a closed pointed full-dimensional convex cone (see \cite{Reznick}) that contains no straight line. %We further call $f\in\mathcal{F}_{n+1,2d}$ a \textit{sum of squares}, if $f=\sum\limits_{i=1}^sg_i^2$ for some $g_1,\ldots,g_s\in\mathcal{F}_{n+1,d}$ and $s\in\N$.
The set of all SOS forms in $\mathcal{F}_{n+1,2d}$, denoted by $\Sigma_{n+1,2d}$, is also a closed pointed full-dimensional convex cone (see \cite{Reznick}). Clearly, $\Sigma_{n+1,2d}\subseteq\mathcal{P}_{n+1,2d}$. It is crucial %for both real algebraic geometry and polynomial optimization 
to determine when the inverse inclusion holds as well. %and thus also $\Sigma_{n+1,2d}$ contains no straight line. 
%We set $\Delta_{n+1,2d}:=\mathcal{P}_{n+1,2d}\backslash\Sigma_{n+1,2d}$. 
%
Hilbert's 1888 Theorem states that $\Sigma_{n+1,2d} = \mathcal{P}_{n+1,2d}$ exactly in the \textit{Hilbert cases} $(n+1,2d)$ with
$n+1=2$ or $2d=2$ or $(n+1,2d)=(3,4)$. Therefore, in any \textit{non-Hilbert case},
%\subsection{General Overview}
%It is crucial for both real algebraic geometry and polynomial optimization to be able to determine if a given positive semi-definite (PSD) $(n+1)$-ary $2d$-ic (i.e., a non-negative form in $n+1$ variables of degree $2d$) can be represented by a finite sum of squares (SOS) of half degree $d$ forms. 
%In 1888, Hilbert \cite{Hilbert} classified all cases $(n+1,2d)$ for which $\Delta_{n+1,2d}=\emptyset$. These are the \textit{Hilbert cases} $(2,2d)$, $(n+1,2)$ and $(3,4)$. In all other cases, the \textit{non-Hilbert cases}, %Hilbert non-constructively verified the existence of PSD-not-SOS $(n+1)$-ary $2d$-ics by a reduction to the \textit{basic non-Hilbert cases} $(4,4)$ and $(3,6)$ via an argument which allows to increase the number of variables and the even degree considered while preserving the PSD-not-SOS property. In any non-Hilbert case, 
%the cone 
$\Sigma_{n+1,2d}\subsetneq\mathcal{P}_{n+1,2d}$. %of PSD $(n+1)$-ary $2d$-ics. 
%It was approximately 80 years after Hilbert's seminal paper that 
	The first explicit examples of a PSD-not-SOS ternary sextic and a PSD-not-SOS quaternary quartic were given by Motzkin \cite{Motzkin} and Robinson \cite{Rob}, respectively. %Soon thereafter, Choi and Lam \cite{CL2, CL} introduced examples of a PSD-not-SOS ternary sextic and a PSD-not-SOS quaternary quartic.
Over the years a significant interest in geometric properties of the cones $\Sigma_{n+1,2d}$ and $\mathcal{P}_{n+1,2d}$ has been developed. For instance, both %$\Sigma_{n+1,2d}$ and $\mathcal{P}_{n+1,2d}$ 
are convex semialgebraic sets. Other important geometric properties include being a spectrahedron and a spectrahedral shadow. A \textit{spectrahedron} is the preimage of the cone of positive semidefinite real matrices under some affine linear map and a \textit{spectrahedral shadow} is the image of a spectrahedron under an affine linear map. In particular, $\Sigma_{n+1,2d}$ is a spectrahedral shadow, %(i.e., the image of a spectrahedron under an affine linear map) 
but $\mathcal{P}_{n+1,2d}$ is not in the non-Hilbert cases (see \cite{Scheiderer}). %In fact, it is known that any spectrahedral shadow is a convex semialgebraic set. In 2006, Nemirovski \cite{Nemirovski} asked if the converse is true as well. In 2009, Helton and Nie \cite{HeltonNie} conjectured that any convex semialgebraic set (over $\mathbb{R}$) is a spectrahedral shadow. During the next decade several results supporting the Helton-Nie conjecture in special cases were given, e.g., \cite{HN2,NetzerSanyal,Scheiderer2}. 
%and, in 2018, Scheiderer \cite{Scheiderer} disproved the Helton-Nie \cite{HeltonNie} that any convex semialgebraic set (over $\mathbb{R}$) is a spectrahedral shadow conjecture by %developing a method to produce convex semialgebraic sets that are not spectrahedral shadows. In particular, he gave the first counterexample by 
%showing that $\mathcal{P}_{n+1,2d}$ is not a spectrahedral shadow in the non-Hilbert cases. %It is thus interesting to investigate at what stage between the cones $\Sigma_{n+1,2d}$ and $\mathcal{P}_{n+1,2d}$ the property of being a spectrahedral shadow is lost. \\
%
%\enlargethispage{\baselineskip}
%\noindent 
%\subsubsection*{\textbf{Construction of a Cone Filtration between $\Sigma_{n+1,2d}$ and $\mathcal{P}_{n+1,2d}$}} \label{Gram}

For $n,d\in\N$, let $k(n,d):=\dim(\mathcal{F}_{n+1,d})-1=\binom{n+d}{n}-1$ and write $k$ instead of $k(n,d)$ whenever $n,d$ are clear from the context. In \cite{GHK, GHK2}, we constructed and examined a specific cone filtration 
\begin{equation}\label{int:eq1} \Sigma_{n+1,2d}=C_0 \subseteq \ldots \subseteq C_n \subseteq C_{n+1} \subseteq \ldots\subseteq C_{k(n,d)-n}=\mathcal{P}_{n+1,2d}\end{equation} along $k(n,d)-n+1$ projective varieties containing the Veronese variety via the Gram matrix method introduced in \cite{CLR}. To this end, we ordered the set $I_{n+1,d}:=\{\alpha\in\N_0^{n+1}\mid\vert\alpha\vert=d\}$ lexicographically. This gave us the ordered set $\{\alpha_0,\ldots,\alpha_{k(n,d)}\}$ and allowed us to enumerate the monomial basis of $\mathcal{F}_{n+1,d}$ along the exponents $\alpha_j=(\alpha_{j,0},\ldots,\alpha_{j,n})$ by setting $m_j(X):=X^{\alpha_j}$ for $j=0,\ldots,k(n,d)$. We furthermore introduced the surjective \textit{Gram map}
\begin{eqnarray*} \mathcal{G} \colon \mathrm{Sym}_{k+1}(\R) & \to & \mathcal{F}_{n+1,2d} \\
	A				& \mapsto & f_A(X):=q_A(m_0(X)\ldots m_k(X))
\end{eqnarray*} and the \textit{(projective) Veronese embedding}
\begin{eqnarray*} V\colon \PP^n & \to & \PP^k \\
\, [x] & \mapsto & [m_0(x):\ldots:m_k(x)].
\end{eqnarray*} We observed \begin{eqnarray}
\label{Kern} \mathcal{G}^{-1}(f)&=& \{A\in\mathrm{Sym}_{k+1}(\R)\mid q_{A-A_f} \mbox{\ vanishes\ on\ } V(\PP^n)\}, \\
%\begin{eqnarray}
\label{CharSOS2}
\Sigma_{n+1,2d}
&=&\{f\in\mathcal{F}_{n+1,2d}\mid \exists\ A\in \mathcal{G}^{-1}(f) \colon q_A\vert_{_{\PP^{k}(\R)}}\geq 0\}, \\
\label{CharPSD} 
\mathcal{P}_{n+1,2d} &=&\{f\in\mathcal{F}_{n+1,2d}\mid \exists\ A\in \mathcal{G}^{-1}(f) \colon q_A\vert_{_{\mathcal{V}(\PP^n)(\R)}}\geq 0\}.
\end{eqnarray}
This motivated us to define the convex cone $$C_W:=\{f\in\mathcal{F}_{n+1,2d}\mid \exists\ A\in \mathcal{G}^{-1}(f) \colon q_A\vert_{_{W(\R)}}\geq 0\}$$ for a set $W\subseteq\PP^k$. Clearly, $\Sigma_{n+1,2d}=C_{\PP^k}$ and $\mathcal{P}_{n+1,2d}=C_{V(\PP^n)}$. Thus, if $V(\PP^n)\subseteq W$, then $C_W$ is an intermediate cone between $\Sigma_{n+1,2d}$ and $\mathcal{P}_{n+1,2d}$. Consequently, $C_W$ is pointed full-dimensional and contains no straight line.

In particular, for $i=0,\ldots,k-n$, we let $$H_i:=\{[z]\in\PP^k\mid\exists\ x\in \C^{n+1}\colon(z_0,\ldots,z_{n+i})=(m_0(x),\ldots,m_{n+i}(x))\}$$ and $V_i\subseteq \PP^k$ be the Zariski closure of $H_i$ in $\PP^k$. %(that is, the smallest projective variety in $\PP^k$ containing $H_i$). 
%This gave us a specific filtration of projective varieties $V(\PP^n)=V_{k-n} \subsetneq \ldots \subsetneq V_0=\PP^k$ with a corresponding specific filtration of sets of real points
%\begin{eqnarray} \label{SetOfRealPointsFiltration} V(\PP^n)(\R)=V_{k-n}(\R) \subsetneq \ldots \subsetneq V_0(\R)=\PP^k(\R)
%\end{eqnarray} in which all inclusions are strict. Thus, for $i=0,\ldots,k-n$, we set
Hence, setting $C_i:=C_{V_i}$, we obtained the cone filtration
\begin{equation} \label{ConeFiltration}
\Sigma_{n+1,2d}=C_0\subseteq \ldots \subseteq C_{k-n}=\mathcal{P}_{n+1,2d}.
\end{equation}

For later purpose, we recall from \cite[Theorem 2.2]{GHK} that, for $i=1,\ldots,k-n$, $V_i$ is the projective closure of the affine variety $$K_i:=\mathcal{V}(q_1(1,Z_1,\ldots,Z_k),\ldots,q_i(1,Z_1,\ldots,Z_k))\subseteq\C^k$$ under the embedding 
$$\begin{array}{cccc} \phi\colon&\C^k&\to&\PP^k \\ &(z_1,\ldots,z_k)&\mapsto& [1:z_1:\ldots:z_k].
\end{array}$$ Here, $q_j(Z):=Z_0Z_{n+j}-Z_{s_j}Z_{t_j}\in\R[Z]:=\R[Z_0,\ldots,Z_k]$ for $$s_j:=\min\{s\in\{1,\ldots,k\}\mid \exists\ t\in\{s,\ldots,k\}\colon \alpha_{s}+\alpha_{t}=\alpha_0+\alpha_{n+j}\}$$ and $t_j\in\{s_j,\ldots,k\}$ such that $\alpha_{s_j}+\alpha_{t_j}=\alpha_0+\alpha_{n+j}$. %Moreover, we define $q_i(Z):=Z_0Z_{n+i}-Z_{s_i}Z_{t_i}$ and recall from \cite[Theorem 2.2]{GHK} that

%It is a priori not clear how many inclusions in (\ref{int:eq1}) are strict and which ones. %This brings us to the following question:
%\begin{center} ($Q$) \textit{Which inclusions in (\ref{ConeFiltration}) are strict?} \end{center}
In \cite{GHK,GHK2}, we determined each strict inclusion in (\ref{int:eq1}) in non-Hilbert cases by establishing %\cite[Theorem A]{GHK}, we showed for $n\geq 3$ that $C_0=\ldots=C_n$. If $n=2$, we even established $C_0=C_1=C_2=C_3$. %Question ($Q$) thus reduces to:
%\begin{center} ($Q^\prime$) \textit{Which inclusions in $C_{n}\subseteq C_{n+1} \subseteq \ldots \subseteq C_{k-n}$ are strict?} \end{center}
%In \cite[Theorem 4.1, 4.2, C and D]{GHK}, we demonstrated that $C_{n+1} \subsetneq \ldots \subsetneq C_{k-n}$ for the cases $(n+1,4)_{n\geq 3}$ and $(n+1,6)_{n\geq 2}$. In the cases $(n+1,4)_{n\geq 3}$, we moreover established $C_n\subsetneq C_{n+1}$. In \cite[Theorem C]{GHK2}, we extended our findings from the non-Hilbert cases of quartics and sextics to any non-Hilbert case $(n+1,2d)$ for $n\geq 2$, $d\geq 4$. Thus, for the non-Hilbert cases, 
\begin{equation}\label{int:separation} \begin{cases} C_0=\ldots=C_n=C_{n+1}\subsetneq \ldots \subsetneq C_{k-n}, & \mbox{if } n=2 \\
	C_0=\ldots=C_n\subsetneq C_{n+1}\subsetneq \ldots \subsetneq C_{k-n}, & \mbox{else} \end{cases}
\end{equation} and showed that none of the strictly separating $C_i$'s is a spectrahedral shadow.
In this paper, we %investigate further properties of the strictly separating intermediate cones $C_i$ in (\ref{int:eq1}) in non-Hilbert cases and formulate some open problems. We 
further examine the $C_i$'s for closure, interior and boundary (see \Cref{Closure}). In \Cref{problems}, we discuss open problems (originating from \cite{Hess}) concerning membership tests for the $C_i$'s and their dual cones. Finally, we consider potential generalizations to cones defined by toric varieties.
\section{Closure, Interior and Boundary} \label{Closure}
%The $\R$-vector spaces $\mbox{Sym}_{k+1}(\R)$, $\mathcal{F}_{n+1,2d}$ and $\mathcal{F}_{k+1,2}$ are finite-dimensional. Hence, 

All norms on the finite-dimensional $\R$-vector spaces $\mbox{Sym}_{k+1}(\R)$, $\mathcal{F}_{n+1,2d}$, respectively, $\mathcal{F}_{k+1,2}$ are equivalent. Hence, they induce the same topology. %Therefore, we may endow each %of the finite-dimensional $\R$-vector spaces $\mbox{Sym}_{k+1}(\R)$, $\mathcal{F}_{n+1,2d}$ and $\mathcal{F}_{k+1,2}$ 
%with some norm $\vert \vert \cdot\vert \vert $ (no matter which one) and treat them as normed $\R$-vector spaces.

\subsection{Closure} Like $\Sigma_{n+1,2d}$ and $\mathcal{P}_{n+1,2d}$, each of the $C_i$'s is also closed as we show below.
\begin{thm} \label{prop:Ciclosed} For $i=0,\ldots,k-n$, $C_i$ is closed.
\end{thm}
	\begin{proof} The Gram map is a linear map between the finite-dimensional normed $\R$-vector spaces $\mathrm{Sym}_{k+1}(\R)$ and $\mathcal{F}_{n+1,2d}$. Therefore, $\mathcal{G}$ is continuous and bounded. The normed $\R$-vector spaces $\mbox{Sym}_{k+1}(\R)$ and $\mathcal{F}_{n+1,2d}$ are furthermore Banach spaces and thus $\mathcal{G}$ is also open by the open mapping theorem (cf.\ \cite[Theorem 6.3]{Ovchinnikov}). We now consider the well-defined bijective linear map \begin{eqnarray*} \overline{\mathcal{G}}\colon \quotient{\mbox{Sym}_{k+1}(\R)}{\ker(\mathcal{G})}& \to & \mathcal{F}_{n+1,2d} \\ \ [A] & \mapsto & \overline{\mathcal{G}}([A]):=\mathcal{G}(A),\end{eqnarray*} where $[A]$ denotes the equivalence class of $A\in\mathrm{Sym}_{k+1}(\R)$ w.r.t.\ $\sim_{\ker(\mathcal{G})}$.
	%
	%\noindent\underline{Claim 1}: $\overline{\mathcal{G}}$ is a well-defined bijective linear map. \\
	%\noindent \textit{Proof}. We recall from \Cref{lem:Gramsurj} that $\mathcal{G}$ is a well-defined surjective linear map. 
	%
	%\noindent \underline{Well-definedness}: For $A,B\in\mbox{Sym}_{k+1}(\R)$ such that $A\sim_{\ker(\mathcal{G})} B$, we see that
	%$$\overline{\mathcal{G}}([A])=\mathcal{G}(A)=\mathcal{G}(B+(A-B))=\mathcal{G}(B)+\mathcal{G}(A-B)=\mathcal{G}(B)=\overline{\mathcal{G}}([B].$$
	%
	%\noindent \underline{Linearity}: For $[A],[B]\in\quotient{\mbox{Sym}_{k+1}(\R)}{\ker(\mathcal{G})}$ and $\lambda\in\R$, the linearity of $\mathcal{G}$ yields $$\overline{\mathcal{G}}([A]+\lambda [B])=\overline{\mathcal{G}}([A+\lambda B])=\mathcal{G}(A+\lambda B)=\mathcal{G}(A)+\lambda \mathcal{G}(B)=\overline{\mathcal{G}}([A])+\lambda\overline{\mathcal{G}}([B]).$$
	%
	%\noindent \underline{Surjectivity}: For $f\in\mathcal{F}_{n+1,2d}$, we fix $A\in\mathcal{G}^{-1}(f)$, which is possible by the surjectivity of $\mathcal{G}$, and conclude for $[A]\in\quotient{\mbox{Sym}_{k+1}(\R)}{\ker(\mathcal{G})}$ that $\overline{\mathcal{G}}([A])=\mathcal{G}(A)=f$.  
	%
	%\noindent \underline{Injectivity}: For $[A]\in\quotient{\mbox{Sym}_{k+1}(\R)}{\ker(\mathcal{G})}$ such that $\overline{\mathcal{G}}([A])=\mathcal{G}(A)$ is the zero form, we deduce $A\in\ker(\mathcal{G})$. Therefore, $A$ is equivalent to the $(k+1)\times(k+1)$ zero matrix $O$ w.r.t.\ the kernel of $\mathcal{G}$. This shows $[A]=[O]$ in $\quotient{\mbox{Sym}_{k+1}(\R)}{\ker(\mathcal{G})}$. \qed 
	%
	Moreover, we interpret the quotient space $\quotient{\mbox{Sym}_{k+1}(\R)}{\ker(\mathcal{G})}$ as the topological space that is obtained by endowing the quotient set $\quotient{\mbox{Sym}_{k+1}(\R)}{\ker(\mathcal{G})}$ with the quotient topology induced by the normed $\R$-vector spaces $\mbox{Sym}_{k+1}(\R)$. Hence, we deduce that $\overline{\mathcal{G}}$ is open since $\mathcal{G}$ is open. %\\
	%\noindent \textit{Proof}. Let $\pi\colon \mathrm{Sym}_{k+1}(\R)\to\quotient{\mbox{Sym}_{k+1}(\R)}{\ker(\mathcal{G})}, A \mapsto [A]$ be the canonical quotient map and $\mathcal{U}\subseteq \quotient{\mbox{Sym}_{k+1}(\R)}{\ker(\mathcal{G})}$ open. Thus, $U:=\pi^{-1}(\mathcal{U})\in\mathrm{Sym}_{k+1}(\R)$ is open. Therefore, also $\mathcal{G}(U)$ is open since $\mathcal{G}$ is an open map. Hence, it suffices to show $\overline{\mathcal{G}}(\mathcal{U})=\mathcal{G}(U)$. 
	%
	%\noindent ($\subseteq$) For $[A]\in\mathcal{U}$, we observe $A\in\pi^{-1}([A])\subseteq\pi^{-1}(\mathcal{U})=U$. Therefore, it follows $\overline{\mathcal{G}}([A])=\mathcal{G}(A)\in\mathcal{G}(U)$. \\
	%($\supseteq$) For $A\in U=\pi^{-1}(\mathcal{U})$, we have $[A]=\pi(A)\in\mathcal{U}$. Thus, $\mathcal{G}(A)=\mathcal{\overline{G}}([A])\in\overline{\mathcal{G}}(\mathcal{U})$. \qed 
	Therefore, $\overline{\mathcal{G}}$ is a bijective open map between two topological spaces and thus closed. 
	
	Keeping this in mind, we interpret $\mathcal{F}_{k+1,2}$ as a metric space by endowing $\mathcal{F}_{k+1,2}$ with the metric induced by the norm of $\mathcal{F}_{k+1,2}$. Moreover, we define $W:=\phi(K_i)\subseteq\PP^k$ and consider the closed set
	$$\mbox{PSD}(W):=\left\{q\in\mathcal{F}_{k+1,2}\mid q\vert_{_{W(\R)}}\geq 0\right\}\subseteq \mathcal{F}_{k+1,2}.$$
	%
	%\noindent\underline{Claim 3}: $\mbox{PSD}(W)$ is closed. \\
	%\noindent \textit{Proof.} According to \Cref{thm:closuresequence}, it suffices to show that the limit of any converging sequence in $\mbox{PSD}(W)$ lies in $\mbox{PSD}(W)$. To this end, we let $g\in\mathcal{F}_{k+1,2}$ be the limit of a converging sequence $(g_n)_{n\in\N}\subseteq\mbox{PSD}(W)$. This implies that the non-negative sequence $(g_n(x))_{n\in\N}\subseteq\R_{\geq 0}$ converges to $g(x)\in\R$ for any $[x]\in W(\R)$ and thus $g(x)\geq 0$ for any $[x]\in W(\R)$. \qed
	The map
	\begin{eqnarray*} Q\colon\mathrm{Sym}_{k+1}(\R) &\to&\mathcal{F}_{k+1,2} \\
				A 							&\mapsto&q_A(Z):=ZAZ^t
	\end{eqnarray*} is linear and hence continuous. %between two finite-dimensional normed $\R$-vector spaces, %$\mathrm{Sym}_{k+1}(\R)$ and $\mathcal{F}_{k+1,2}$. Hence, 
 	Thus, $\mathfrak{W}:=Q^{-1}(\mbox{PSD}(W))\subseteq\mathrm{Sym}_{k+1}(\R)$ is closed. Moreover, we consider the well-defined set $$\mathcal{W}:=\left\{[A]\in\quotient{\mbox{Sym}_{k+1}(\R)}{\ker(\mathcal{G})}\mid \exists B\in\ker(\mathcal{G})\colon A+B\in \mathfrak{W}\right\}$$ 
	%\noindent\underline{Claim 4}: $\mathcal{W}\subseteq\quotient{\mbox{Sym}_{k+1}(\R)}{\ker(\mathcal{G})}$ is a well-defined closed set. \\
	%\noindent\textit{Proof.} %For $A\in\mathrm{Sym}_{k+1}(\R)$ with $B\in\ker(\mathcal{G})$ such that $A+B\in\fW$, we choose $C\in\mathrm{Sym}_{k+1}(\R)$ such that $A\sim_{\ker(\mathcal{G})}C$ and set $D:=(A-C)+B\in\mathrm{Sym}_{k+1}(\R)$. Moreover, we denote the zero form in $\mathcal{F}_{n+1,2d}$ by $0$ and deduce $$\mathcal{G}(D)=\mathcal{G}((A-C)+B)=\mathcal{G}(A-C)+\mathcal{G}(B)=0+0=0$$ from the linearity of $\mathcal{G}$. This shows $D\in\ker(\mathcal{G})$. Furthermore, we see $$C+D=C+((A-C)+B)=A+B \in\fW.$$ Therefore, it remains 
	which is closed if and only if $\left\{A\in\mbox{Sym}_{k+1}(\R)\mid [A]\in \mathcal{W}\right\}\subseteq\mathrm{Sym}_{k+1}(\R)$ is closed. Hence, in order to show that $\mathcal{W}$ is closed, it suffices to prove that $\left\{A\in\mbox{Sym}_{k+1}(\R)\mid [A]\in \mathcal{W}\right\}\subseteq\mathrm{Sym}_{k+1}(\R)$ is closed. We do so in two steps. \\
	
	\noindent \underline{Claim 1}: $\{A\in\mbox{Sym}_{k+1}(\R)\mid [A]\in \mathcal{W}\}=\mathfrak{W}+\ker(\mathcal{G})$. \\
	\noindent \textit{Proof.} ($\subseteq$) For $A\in\mbox{Sym}_{k+1}(\R)$ with $[A]\in \mathcal{W}$, we let $B\in\ker(\mathcal{G})$ be such that $C:=A+B\in\mathfrak{W}$ and conclude $A=C-B\in \mathfrak{W}-\ker(\mathcal{G})=\mathfrak{W}+\ker(\mathcal{G})$. \\
	($\supseteq$) For $A\in\mathfrak{W}$, $B\in\ker(\mathcal{G})$, we observe $[A+B]=[A]+[B]=[A]\in\mathcal{W}$. This proves Claim 1. \\
	
	\noindent \underline{Claim 2}: $\mathfrak{W}+\ker(\mathcal{G})$ is closed. \\
	\noindent \textit{Proof.} We observe that $\mathfrak{W}$ and $\ker(\mathcal{G})$ are non-empty cones in $\mathrm{Sym}_{k+1}(\R)$ and
	%
	%Indeed, the $(k+1)\times(k+1)$ zero matrix is contained in both $\fW$ and $\ker(\mathcal{G})$. Moreover, for $A,B\in \fW$ and $\lambda\geq 0$, we have $Q(A+\lambda B)=Q(A)+\lambda Q(B)=q_A+\lambda q_B$ by the linearity of $Q$ and we know that $q_A$ and $q_B$ are locally PSD on $W(\R)$. Since $\lambda$ is assumed to be non-negative, we thus obtain that $q_A+\lambda q_B$ is locally PSD on $W(\R)$. This shows $A+\lambda B \in\fW$. 
	%Likewise, for $A,B\in \ker(\mathcal{G})$ and $\lambda\geq 0$, we denote the zero form in $\mathcal{F}_{n+1,2d}$ by $0$ and deduce $\mathcal{G}(A+\lambda B)=\mathcal{G}(A)+\lambda \mathcal{G}(B)=0 + \lambda \cdot 0 =0$ from the linearity of $\mathcal{G}$. This shows $A+\lambda B\in\ker(\mathcal{G})$.
	recall that $\mathrm{Sym}_{k+1}(\R)$ is a locally convex Hausdorff space that is locally compact. Since we know that $\mathfrak{W}$ and $\ker(\mathcal{G})\subseteq\mathrm{Sym}_{k+1}(\R)$ are closed, we thus deduce that $\mathfrak{W}$ and $\ker(\mathcal{G})$ are locally compact. Moreover, the recession cones of the cones $\mathfrak{W}$ and $\ker(\mathcal{G})$ coincide with $\mathfrak{W}$ and $\ker(\mathcal{G})$, respectively. By an application of Dieudonné's theorem (refer to \cite[Theorem 1.1.8]{Zaline} for an interpretation of Dieudonné's original results \cite{Dieu}), it thus suffices to show that $\mathfrak{W}\cap\ker(\mathcal{G})$ is a linear space. To this end, we claim the following:
	\medskip
	%to conclude that $\mathfrak{W}-\ker(\mathcal{G})=\mathfrak{W}+\ker(\mathcal{G})$ is closed. 
	\begin{enumerate}
		\item[(i)] If $i=k-n$, then $\mathfrak{W}\cap\ker(\mathcal{G})=\ker(\mathcal{G})$.
		\item[(ii)] \vspace*{0.1cm} If $i<k-n$, then $\mathfrak{W}\cap\ker(\mathcal{G})$ coincides with the set
		$$\hspace*{1.2cm} \left\{A:=(a_{s,t})_{0\leq s,t\leq k}\in \ker(\mathcal{G}) \mid \forall t\geq n+i+1 \colon a_{s,t}=0 \mbox{ for } s=0,\ldots,t \right\}.$$
	\end{enumerate}
	\medskip
	For subclaim (i), it suffices to show $\ker(\mathcal{G})\subseteq\mathfrak{W}$ and we know $$\mathfrak{W}=\left\{A\in\mathrm{Sym}_{k+1}(\R) \mid q_A\vert_{_{W(\R)}}\geq 0\right\}=\left\{A\in\mathrm{Sym}_{k+1}(\R) \mid q_A\vert_{_{V_{k-n}(\R)}}\geq 0\right\}$$ by construction since $V_{k-n}$ is the Zariski closure of $W=\phi(K_{k-n})$. Moreover, $V_{k-n}=V(\PP^n)$ and, therefore, we see $$\ker(\mathcal{G})=\{A\in\mathrm{Sym}_{k+1}(\R) \mid q_A \mathrm{\ vanishes\ on\ } V(\PP^n)\}\subseteq \mathfrak{W}.$$ This proves subclaim (i). \\ %=\left\{A\in\mathrm{Sym}_{k+1}(\R) \mid q_A\vert_{_{V(\PP^n)(\R)}}\geq 0 \right\}$$
	
	\noindent For subclaim (ii), we have to verify two inclusions. 
	
	To show ($\subseteq $), we let $A:=(a_{s,t})_{0\leq s,t\leq k}\in\mathfrak{W}\cap\ker(\mathcal{G})$ be arbitrary but fixed and observe that $a_{k,k}$ is the coefficient of $X_n^{2d}$ in $\mathcal{G}(A)$. Hence, $a_{k,k}=0$ follows from $A\in\ker(\mathcal{G})$. Moreover, since $[m_0(1,\boldsymbol{x}):\ldots:m_{k-1}(1,\boldsymbol{x}):y]\in \phi(K_{k-n-1})(\R) \subseteq \phi(K_i)(\R)$ for any $\boldsymbol{x}\in\R^{n}$ and any non-zero $y\in\R$, we deduce from $A\in\mathfrak{W}$ that
	\enlargethispage{\baselineskip}
	\begin{eqnarray*}
		q(\boldsymbol{X},Y)&:=&q_A(m_0(1,\boldsymbol{X}),\ldots,m_{k-1}(1,\boldsymbol{X}),Y) \\
		&=&q_A(m_0(1,\boldsymbol{X}),\ldots,m_{k-1}(1,\boldsymbol{X}),m_k(1,\boldsymbol{X}))\\
		&& -2\sum\limits_{s=0}^{k-1} a_{s,k}m_s(1,\boldsymbol{X})m_{k}(1,\boldsymbol{X})+\left(2\sum\limits_{s=0}^{k-1} a_{s,k}m_s(1,\boldsymbol{X})\right)Y
	\end{eqnarray*}
	is PSD. Consequently, $a_{i,k}=0$ for $i=0,\ldots,k-1$ in order to avoid the potential linearity of $q$ in $Y$. If $i=k-n-1$, then we are done. Otherwise, we iterate the above argument for $t=k-1,\ldots,n+i+1$ and conclude $a_{s,t}=0$ for $t=n+i+1,\ldots,k-n$ and $s=0,\ldots,t$. 
	
	To show ($\supseteq$), for $A:=(a_{s,t})_{0\leq s,t\leq k}\in \ker(\mathcal{G})$ such that $a_{s,t}=0$ for $t=n+i+1,\ldots,k$ and $s=0,\ldots,t$, we observe for $\boldsymbol{x}\in\R^n$ and $z_{n+i+1},\ldots,z_k\in\R$ that 
	\begin{eqnarray*}
		q_A(m_0(1,\boldsymbol{x}),\ldots,m_{n+i}(1,\boldsymbol{x}),z_{n+i+1},\ldots,z_k)&=&q_A(m_0(1,\boldsymbol{x}),\ldots,m_k(1,\boldsymbol{x}))\\
		&=&\mathcal{G}(A)(1,\boldsymbol{x})= 0.
	\end{eqnarray*} Hence, $A\in\mathfrak{W}$ since $W=\phi(K_i)$. This proves subclaim (ii). \\
	
	We thus see that if $i=k-n$, then $\mathfrak{W}\cap\ker(\mathcal{G})=\ker(\mathcal{G})$ is a linear space. Otherwise, if $i<k-n$, then we compute for $\lambda\in\R$, $A:=(a_{s,t})_{0\leq s,t\leq k}$ and $B:=(b_{s,t})_{0\leq s,t\leq k}\in \ker(\mathcal{G})$ such that $a_{s,t}=b_{s,t}=0$ for $t=n+i+1,\ldots,k$ and $s=0,\ldots,t$ that $$(A+\lambda B)_{s,t}=a_{s,t}+\lambda b_{s,t}=0+\lambda \cdot 0 = 0$$ for $t=n+i+1,\ldots,k$ and $s=0,\ldots,t$. Hence, by (ii), we deduce that $\mathfrak{W}\cap\ker(\mathcal{G})$ also is a linear space if $i<k-n$. This completes our proof of Claim 2. \\
	
	To conclude the proof, we lastly oberserve $C_i=\overline{\mathcal{G}}(\mathcal{W})$ %\\
	%\noindent\textit{Proof.} ($\subseteq$) For $[A]\in\mathcal{W}$, we fix $B\in\ker(\mathcal{G})$ such that $C:=A+B\in\fW$ and set $f:=\mathcal{G}(A)=\overline{\mathcal{G}}([A])$. Hence, $C$ is a Gram matrix associated to $f$ and $q_C=Q(C)$ is locally PSD on $W(\R)$ since $C\in\fW$. We have $W=\phi(K_i)$ by construction and thus $\overline{\mathcal{G}}(A)=f\in C_W=C_{\phi(K_i)}=C_i$ follows using \Cref{cor:extendingCi}. \\
	%($\supseteq$) For $f\in C_i$, we fix $A\in\mathcal{G}^{-1}(f)$ such that $q_A$ is locally PSD on $V_i(\R)$ and observe that $Q(A)=q_A$ is especially locally PSD on the subset $\phi(K_i)(\R)$ of $V_i(\R)$. This shows $A\in\fW$ which implies $[A]\in\mathcal{W}$. We conclude $f=\mathcal{G}(A)=\overline{\mathcal{G}}([A])\in\overline{\mathcal{G}}(\mathcal{W})$. \qed 
	which yields that $C_i$ is closed as the image of the closed set $\mathcal{W}$ under the closed map $\overline{\mathcal{G}}$.%, we thus deduce that $\overline{\mathcal{G}}(\mathcal{W})=C_i$ is closed.
\end{proof}
%\begin{rem} It follows that the cones $\mathcal{P}_{n+1,2d}=C_{k-n}$ and $\Sigma_{n+1,2d}=C_0$ are closed, which agrees with common knowledge. %(cf.\ \Cref{lem:convexpropPSDSOS}).
%\end{rem}

\begin{cor} For $i=0,\ldots,k-n$, $C_i$ is the convex hull of its extreme rays.
\end{cor}
	\begin{proof} \Cref{prop:Ciclosed} and Krein--Milman's theorem (cf.\ \cite{KM}) together yield the assertion since $C_i$ is a closed convex cone that contains no straight line.
	\end{proof}

\subsection{Interior and Boundary}
In order to describe the interior of $C_i$, we need to introduce a non-negativity notion that is stronger than the (local) positive semidefinite property of a given quadratic form. That is, we call $q\in\mathcal{F}_{k+1,2}$ \textit{(locally) positive} or \textit{(locally) positive definite} on $W(\R)\subseteq\PP^k$, if $q(z)>0$ for all $[z]\in W(\R)$ and write $q\vert_{_{W(\R)}}>0$. In the special case that $W(\R)=\PP^k(\R)$, we say that $q$ is \textit{(globally) positive} or \textit{(globally) positive definite (PD)} and write $q>0$.

\begin{thm} \label{prop:interiorCi} For $i=0,\ldots,k-n$, the interior of $C_i$ is given by $$\accentset{\circ}{C_i}=\left\{f\in\mathcal{F}_{n+1,2d} \mid \exists A\in\mathcal{G}^{-1}(f) \colon q_A\vert_{_{V_i(\R)}}>0\right\}.$$
\end{thm}
\begin{proof} At the beginning of the proof of \Cref{prop:Ciclosed}, we observed that $\mathcal{G}$ is a homeomorphism between $\mathrm{Sym}_{k+1}(\R)$ and $\mathcal{F}_{n+1,2d}$. %\footnote{That is, $\mathcal{G}$ is a bijective  continuous open map (cf.\ \Cref{defn:homeo}).} 
Thus, $C_i=\mathcal{G}(\mathfrak{W})$ for $$\mathfrak{W}:=\left\{A\in\mathrm{Sym}_{k+1} \mid q_A\vert_{_{V_i(\R)}}\geq 0\right\}.$$ Hence, $\accentset{\circ}{C_i}=(\accentset{\circ}{\mathcal{G}(\mathfrak{W})})=\mathcal{G}(\accentset{\circ}{\mathfrak{W}})$. Therefore, it suffices to show $$\accentset{\circ}{\mathfrak{W}}=\{A\in\mathrm{Sym}_{k+1}(\R) \mid q_A\vert_{_{V_i(\R)}}>0\}.$$
	
	\noindent ($\subseteq$) For $A\in\accentset{\circ}{\mathfrak{W}}$, there exists some $\varepsilon>0$ such that the open ball $B_{\varepsilon}(A)$ of radius $\varepsilon$ with center $A$ in $\mathrm{Sym}_{k+1}(\R)$ is contained in $\mathfrak{W}$. We fix such $\varepsilon>0$ and let $\lambda>0$ be sufficiently small such that $\vert \vert -\lambda I_{k+1} \vert \vert = \lambda \vert \vert I_{k+1} \vert \vert <\varepsilon$, where $I_{k+1}$ denotes the $(k+1)\times(k+1)$ identity matrix. Hence,
	$\vert \vert (A-\lambda I_{k+1})-A \vert \vert <\varepsilon$ implies $A-\lambda I_{k+1}\in B_{\varepsilon}(A)\subseteq\mathfrak{W}$ which yields that $$q_{A-\lambda I_{k+1}}(Z)=q_A(Z)-\lambda q_{I_{k+1}}(Z)=q_A(Z)-\lambda\sum\limits_{j=0}^k Z_j^2$$ is locally PSD on $V_i(\R)$. Since $\lambda\sum\limits_{j=0}^k z_j^2>0$ for any $[z]\in V_i(\R)$, we thus conclude $q_A(z)>0$ for any $[z]\in V_i(\R)$. \\
	
	\noindent ($\supseteq$) We set $\mathcal{B}$ to be the closed unit ball in $\R^{k+1}$ and recall that $\mathcal{B}$ is especially compact. Moreover, we specify the norm on $\mathrm{Sym}_{k+1}(\R)$ to be given by
	\begin{eqnarray*} \vert \vert \cdot \vert \vert \colon \mathrm{Sym}_{k+1}(\R) & \to & \R_{\geq 0} \\ A &\mapsto & \max\limits_{z\in \mathcal{B}}\vert q_A(z)\vert
	\end{eqnarray*}
	and interpret $V_i(\R)$ as an affine set by setting $W:=\{z\in\R^{k+1} \mid [z]\in V_{i}(\R)\}$. 
	
	We claim that $W$ is closed.
	%\noindent \underline{Claim 2}: $W$ is closed. \\
	%\textit{Proof.} According to \Cref{thm:closuresequence}, 
	Indeed, %it suffices to show that the limit of any converging sequence in $W$ lies in $W$. To this end, we 
	let $\left(z^{(n)}\right)_{n\in\N}\subseteq W$ be a converging sequence with limit $z\in\R^{k+1}$ and choose $S\subseteq \C[Z]$ such that $V_i=\mathcal{V}(S)$. For $n\in\N$, $z^{(n)}\in W$ implies $\left[z^{(n)}\right]\in V_i(\R)=\mathcal{V}(S)(\R)$ and thus $f\left(z^{(n)}\right)=0$ for any $f\in S$ and any $n\in\N$. Therefore, $$f(z)=f\left(\lim\limits_{n\to \infty} z^{(n)}\right)=\lim\limits_{n\to\infty} f\left(z^{(n)}\right)=0$$ follows for any $f\in S$. Hence, $[z]\in\mathcal{V}(S)(\R)=V_i(\R)$, respectively $z\in W$. 
	
	Consequently, the intersection $\mathcal{B}\cap W$ of the compact set $\mathcal{B}$ with the closed set $W$ is compact. Hence, we deduce for arbitrary but fixed $A\in\mathrm{Sym}_{k+1}(\R)$ with $q_A\vert_{_{V_i(\R)}}>0$ that $q_A$ attains a minimum on $\mathcal{B}\cap W$. This allows us to fix some $\varepsilon>0$ such that $$\min\limits_{z\in \mathcal{B}\cap W} q_A(z)>\varepsilon > 0$$ and we let $B_{\varepsilon}(A)$ be the open ball of radius $\varepsilon$ with center $A$ in $\mathrm{Sym}_{k+1}(\R)$.
	
	Lastly, we claim that $B_{\varepsilon}(A)\subseteq \mathfrak{W}$. Indeed, for $B\in B_{\varepsilon}(A)$, it holds $\varepsilon > \vert \vert B-A \vert \vert = \max\limits_{z\in\mathcal{B}}\vert q_{B-A}(z)\vert$ and we have to show that $q_B$ is locally PSD on $V_i(\R)$. To this end, we let $[z]\in V_i(\R)$ be arbitrary but fixed and deduce $z\in W$. Moreover, we choose a sufficiently small $\lambda>0$ such that $\vert \lambda z \vert \leq 1$ and observe $\lambda z \in \mathcal{B}\cap W$. %since $[\lambda z] = [z]\in V_i(\R)$. 
	Recalling $\min\limits_{y\in \mathcal{B}\cap W} q_A(y)>\varepsilon > 0$, we thus conclude
	\begin{eqnarray*} q_B(\lambda z)&=&q_{A+(B-A)}(\lambda z) \\
		&=&q_A(\lambda z)+q_{B-A}(\lambda z) \\
		&\geq&\min\limits_{y\in \mathcal{B}\cap W} q_A(y)+q_{B-A}(\lambda z) \\
		&>&\varepsilon+q_{B-A}(\lambda z) \\
		&=&\varepsilon-(-q_{B-A}(\lambda z)).
	\end{eqnarray*} 
	If $q_B(\lambda z)<0$, then the above implies $-q_{B-A}(\lambda z)>\varepsilon>0$. Thus, $$\vert \vert B-A \vert \vert =\max\limits_{z\in \mathcal{B}}\vert q_{B-A}(z)\vert=\max\limits_{z\in \mathcal{B}}\vert -q_{B-A}(z)\vert\geq -q_{B-A}(\lambda z) > \varepsilon$$ which contradicts $B\in B_{\varepsilon}(A)$. Hence, $q_B(\lambda z)<0$ cannot be true and it follows $0\leq q_B(\lambda z)=\lambda^2q_B(z)$. Thus, $q_B(z)>0$.
	
	Altogether, we therefore showed that any $A\in\mathrm{Sym}_{k+1}(\R)$ such that $q_A\vert_{_{V_i(\R)}}>0$ lies in the interior of $\mathfrak{W}$.
\end{proof}

\begin{exmp} \textsc{The Interior of $\mathcal{P}_{n+1,2d}$} \\
	It is known (cf.\ \cite[Theorem 3.14]{Rez2}) that %the interior of the cone $\mathcal{P}_{n+1,2d}$ is given by 
	$$\accentset{\circ}{\mathcal{P}}_{n+1,2d}=\left\{f\in\mathcal{F}_{n+1,2d} \mid \forall x\in\R^{n+1}, x\neq 0\colon f(x)>0\right\}.$$ Since $V_{k-n}(\R)=V(\PP^n(\R))$, from \Cref{prop:interiorCi} above, we also have %the interior of $\mathcal{P}_{n+1,2d}=C_{k-n}$ is given by
	 $$\accentset{\circ}{\mathcal{P}}_{n+1,2d}=\accentset{\circ}{C}_{k-n}=\left\{f\in\mathcal{F}_{n+1,2d} \mid \exists A\in\mathcal{G}^{-1}(f) \colon q_A\vert_{_{V(\PP^n(\R))}}>0\right\}.$$ Let us show that these two descriptions of $\accentset{\circ}{\mathcal{P}}_{n+1,2d}$ are indeed equivalent.
	
	Let $f\in\mathcal{F}_{n+1,2d}$ be such that $f(x)>0$ for all non-zero $x\in\R^{n+1}$ and fix $A\in\mathcal{G}^{-1}(f)$. Hence, we see $q_A(m_0(x),\ldots,m_k(x))=f(x)>0$ for all non-zero $x\in\R^{n+1}$ which shows that $q_A$ is locally PD on $V((\PP^n(\R))$.
	
	Vice versa, let $f\in\mathcal{F}_{n+1,2d}$ be such that there exists some $A\in\mathcal{G}^{-1}(f)$ for which $q_A$ is locally PD on $V(\PP^n(\R))$. For such a choice of $A\in\mathcal{G}^{-1}(f)$, we conclude $f(x)=q_A(m_0(x),\ldots,m_k(x))>0$ for any non-zero $x\in\R^{n+1}$ since we know that $[m_0(x):\ldots:m_k(x)]=V([x])\in V(\PP^n(\R))$ for any non-zero $x\in\R^{n+1}$.  \qed
\end{exmp}

\begin{exmp} \textsc{The Interior of $\Sigma_{n+1,2d}$} \\
	It is known (cf.\ \cite[Proposition 5.5.]{CLR}) that 
	%It is known that the interior of the cone $\Sigma_{n+1,2d}$ is given by 
	$$\accentset{\circ}{\Sigma}_{n+1,2d}=\left\{f\in\mathcal{F}_{n+1,2d} \mid \exists A\in\mathcal{G}^{-1}(f) \colon A \mbox{ is non-singular and } q_A\geq 0\right\}.$$ Since $V_{0}(\R)=\PP^k(\R)$, from \Cref{prop:interiorCi} above, we also have 
	$$\accentset{\circ}{\Sigma}_{n+1,2d}=\accentset{\circ}{C}_{0}=\left\{f\in\mathcal{F}_{n+1,2d} \mid \exists A\in\mathcal{G}^{-1}(f) \colon q_A>0\right\}.$$ Let us show that these two descriptions of $\accentset{\circ}{\Sigma}_{n+1,2d}$ are indeed equivalent.
	
	Let $f\in\mathcal{F}_{n+1,2d}$ be such that there exists some non-singular $A\in\mathcal{G}^{-1}(f)$ for which $q_A$ is PSD. For such a choice of $A\in\mathcal{G}^{-1}(f)$, we deduce that all eigenvalues of $A$ are non-negative real numbers since $q_A\geq 0$ implies that the matrix $A\in\mathrm{Sym}_{k+1}(\R)$ is positive semidefinite. Moreover, $0$ is not an eigenvalue of $A$ since $A$ is assumed to be non-singular. Altogether, we thus know that any eigenvalue of $A$ is a positive real number. Therefore, the matrix $A$ is positive definite which implies that $q_A$ is PD.
	
	Vice versa, let $f\in\mathcal{F}_{n+1,2d}$ be such that there exists some $A\in\mathcal{G}^{-1}(f)$ for which $q_A$ is PD. For such a choice of $A\in\mathcal{G}^{-1}(f)$, we especially know that $q_A$ is PSD and we moreover deduce that the matrix $A\in\mathrm{Sym}_{k+1}(\R)$ is positive definite since $q_A>0$. Therefore, we know that any eigenvalue of $A$ is a positive real number and thus especially non-zero. We conclude that $A$ is non-singular. \qed
\end{exmp}

\begin{cor} \label{cor:Ciboundary} For $i=0,\ldots,k-n$, the boundary of $C_i$ is given by $$\partial C_i=\left\{f\in C_i \mid \forall A\in\mathcal{G}^{-1}(f)\ \exists [z]\in V_i(\R) \colon q_B(z)\leq 0\right\}.$$
\end{cor}
\begin{proof} \Cref{prop:Ciclosed} states that $C_i$ is closed and thus $\partial C_i=C_i\backslash\accentset{\circ}{C_i}$. The assertion therefore follows from \Cref{prop:interiorCi}. %according to which \[\accentset{\circ}{C_i}=\left\{f\in\mathcal{F}_{n+1,2d} \mid \exists A\in\mathcal{G}^{-1}(f) \colon q_A\vert_{_{V_i(\R)}}>0\right\}. \qedhere\]
\end{proof}

%\textcolor{red}{Computational value: In addition to the geometric description of $C_i$'s (incl. faces, extreme rays, convex hull, etc.), the computational value of $C_i$'s (incl. being Spectrahehdral shadow) and optimizational perspective of $C_i$'s, both seems to be interesting.}

\section{Open Problems} \label{problems}
In this section, we pose some questions concerning membership tests for $C_i$'s, investigations of the dual cone $$C_i^\vee:=\{L\in Hom(\R[X],\R) \mid \forall f\in C_i\colon L(f)\geq 0 \}$$ of $C_i$ and generalizations of our considerations about $C_i$'s to cones along toric varieties. Membership tests are in particular crucial from optimization point of view for evaluating the computational value of $C_i$'s whereas the dual cones $C_i^\vee$'s are related to the moment problem from functional analysis.

\subsection{Membership Tests for $\bf{C_i}$'s} \label{Sec:ConeProp}

From an optimizational perspective $\Sigma_{n+1,2d}$ is more valuable than $\mathcal{P}_{n+1,2d}$ since the membership for $\Sigma_{n+1,2d}$ can be efficiently tested. This is because for $f\in\mathcal{F}_{n+1,2d}$, $f$ lies in $\Sigma_{n+1,2d}$ if and only if there exists an associated Gram matrix $A\in\mathcal{G}^{-1}(f)$ with a corresponding quadratic form $q_A$ that is locally PSD on $\PP^k(\R)$. The latter is equivalent to $A$ being positive semidefinite. It is well-understood when a given symmetric matrix with real entries is positive semidefinite. We recall some characterizations of this property below (cf.\ \cite[0.2.1 Proposition]{Marshall}).

\begin{prop} \label{prop:propPSDmatrix} \index{positive semidefinite matrix}For $l\in\N$, $A\in\mathrm{Sym}_l(\R)$, the following are equivalent:
	\begin{enumerate}
		\item[(i)] $A$ is positive semidefinite.
		\item[(ii)] Every eigenvalue of $A$ is non-negative.
		\item[(iii)] Each principal minor of $A$ is non-negative.
		\item[(iv)] There exists some $l\times l$ matrix $U$ with real entries such that $A=U^tU$.
		\item[(v)] There exist $s\in\N$, $y^{(1)},\ldots,y^{(s)}\in\R^l$ and $\lambda_1,\ldots,\lambda_s\geq 0$ such that $$A=\sum\limits_{i=1}^s\lambda_i\left(y^{(i)}\right)^t\left(y^{(i)}\right).$$
	\end{enumerate}
\end{prop}

Thus, verification of whether a given $A\in\mathrm{Sym}_{k+1}(\R)$ is positive semidefinite by evaluating the corresponding $q_A$ in all $x\in\R^{k+1}$, as done in our setting, can be replaced by testing for any one of the \textit{intrinsic} properties (ii) -- (v) of \Cref{prop:propPSDmatrix}. Therefore, for the cones $C_0,\ldots,C_n$ and also $C_{n+1}$ if $n=2$, corresponing to the varieties of minimal degree $V_0,\ldots,V_n$ and $V_{n+1}$ if $n=2$, membership can be efficiently tested by the verification of any one of the intrinsic properties (ii) -- (v) for representing Gram matrices.

The downside of replacing membership tests to $\mathcal{P}_{n+1,2d}$ by membership tests to $\Sigma_{n+1,2d}$ is that not every PSD $(n+1)$-ary $2d$-ic is SOS and thus membership tests in $\Sigma_{n+1,2d}$ do not suffices to decide whether a given form is PSD in all cases. In an attempt to provide certificates of the PSD property for more $(n+1)$-ary $2d$-ics than those that are SOS, we therefore propose to test membership in the larger strictly separating intermediate cones $C_i$'s in (\ref{int:separation}). This brings us to the following task.

\begin{customqu}{1} \label{prop1} Let $C_i$ be a strictly separating intermediate cone in (\ref{int:separation}). Provide a membership test for $C_i$, preferably by an intrinsic property of a representing Gram matrix.
\end{customqu}

The \Cref{prop1} above has to be addressed with tools other than the ones coming from semidefinite programming since no strictly separating intermediate cone in (\ref{int:separation}) is a spectrahedral shadow by \cite[Theorem D]{GHK2} and with that not a feasible region of a semidefinite programming problem. However, there might be some other method which could provide numerically efficient membership tests for our distinguished intermediate cones $C_i$'s demonstrating their computational values.
%\textcolor{red}{%The geometric description of the $C_i$'s as the convex hull of the extreme rays was given in the previous section, additionally %(incl. extreme rays, convex hull, etc.), 
	%It is interesting to investigate the computational values of $C_i$'s. %In \cite{GHK2} we proved that no strictly separating $\Sigma_{n+1,2d}\subsetneq C_i \subsetneq \mathcal{P}_{n+1,2d}$ in a non-Hilbert case is a spectrahedral shadow and thus no membership test can be provided via semidefinite programming. 
%From an optimizational perspective, the problem of finding membership tests for the intermediate cones $C_i$'s remains open.}

\subsection{Dual Cones of $\bf{C_i}$'s}  \label{Sec:Dual}
The bidual cone $(C_i^\vee)^\vee$ coincides with $C_i$ for $i=0,\ldots,k-n$ since $C_0,\ldots,C_{k-n}$ are closed. Hence, an investigation of the cone filtration (\ref{ConeFiltration}) can equivalently be replaced by that of the dual cone filtration
$$\Sigma_{n+1,2d}^\vee \supseteq C_0^\vee \supseteq \ldots \supseteq C_{k-n}^\vee=\mathcal{P}_{n+1,2d}^\vee$$ since $C_i\subsetneq C_{i+1}$ if and only if $C_i^\vee \supsetneq C_{i+1}^\vee$ for $i=0,\ldots,k-n-1$. Let us relate this dual cone filtration to the %\textit{classical $n$-dimensional moment problem}, from functional analysis, which asks if there exists a non-negative Borel measure $\mu$ on $\R^n$ such that $$L(g)=\int g\mathrm{d}\mu$$ for a given linear functional $L\colon\R[\boldsymbol{X}]\to\R$ and for all $g\in\R[\boldsymbol{X}]$. In this paper, we especially focused on forms of a fixed even degree or, equivalently in a dehomogenized setting, on polynomials up to a fixed even degree. Therefore, we are in the context of the 
\textit{truncated moment problem} from functional analysis as elaborated below. We refer to \cite{CGIK, Marshall} for an overview on the (truncated) moment problem.

\subsubsection*{\textbf{Truncated Moment Problem}} This problem has been extensively studied by Curto--Fialkow \cite{CF1, CF2, CF3}, Fialkow--Nie \cite{FN} and many more in the classical setting and by Curto et al.\ \cite{CGIK} for unital commutative $\R$-algebras.
It asks if there exists a non-negative Borel measure $\mu$ on $\R^n$ such that 
$$L(\boldsymbol{X})=\int \boldsymbol{X}^{(a_1,\ldots,a_n)} \mathrm{d}\mu(\boldsymbol{X})$$ for a given sequence $(y_a)_{a\in I_{n+1,d}}\subseteq \R$ and for all $a\in I_{n+1,d}$. We call $(y_a)_{a\in I_{n+1,d}}$ a \textit{truncated moment sequence} and say that this sequence \textit{admits a measure} if the truncated moment problem can be affirmatively answered with a Borel measure $\mu$. Any such $\mu$ is then called a \textit{representing measure} of $(y_a)_{a\in I_{n+1,d}}$ and we denote the set of all truncated moment sequences $(y_a)_{a\in I_{n+1,d}}$ that admit a measure by $\mathcal{R}_{n,d}$. Moreover, we set $\R[\boldsymbol{X}]_{\leq d}$ to be the $\R$-vector space of polynomials in $\boldsymbol{X}$ of degree at most $d$ with coefficients in $\R$ and observe that a truncated moment sequence $y:=(y_a)_{a\in I_{n+1,d}}\subseteq \R$ induces the \textit{Riesz functional}
$$\begin{array}{rccc} \mathcal{L}_y \colon& \R[\boldsymbol{X}]_{\leq d} &\to & \R \\ 
	&\sum\limits_{a\in I_{n+1,d}} g_a\boldsymbol{X}^{(a_1,\ldots,a_n)} & \mapsto & \sum\limits_{a\in I_{n+1,d}} g_ay_a.
\end{array}$$ Therefore, the truncated moment problem in particular asks if for a given linear functional $L\colon\R[\boldsymbol{X}]_{\leq d}\to\R$, there exists a non-negative Borel measure $\mu$ on $\R^n$ such that, for all $g\in\R[\boldsymbol{X}]_{\leq d}$, $$L(g)=\int g\mathrm{d}\mu.$$ 

A condition that is necessary for a truncated moment sequence $y:=(y_a)_{a\in I_{n+1,d}}$ to admit a measure is, for example, that the Riesz functional $\mathcal{L}_y$ is non-negative in any positive semidefinite $g\in\R[\boldsymbol{X}]_{\leq d}$. That is, $\mathcal{L}_y$ must lie in the dual cone of the cone $\mathcal{P}_{n,\leq d}$ of positive semidefinite $g\in\R[\boldsymbol{X}]_{\leq d}$. In fact, by \cite[Theorem 2.2.]{FN2}, we have $$\overline{\mathcal{R}_{n,d}}=\left\{y:=(y_a)_{a\in I_{n+1,2d}}\subseteq\R\mid \mathcal{L}_y\in \mathcal{P}_{n,\leq d}^\vee\right\}\simeq \mathcal{P}_{n,\leq d}^\vee.$$

\subsubsection*{\textbf{Homogeneous Truncated Moment Problem}} To transfer the above consideration into a homogeneous setting, we observe that a truncated moment sequence $y:=(y_a)_{a\in I_{n+1,d}}\subseteq \R$ admits a measure if and only if there exists a non-negative Borel measure $\mu$ supported on the $n$-dimensional unit sphere $\mathbb{S}^n$ in $\R^{n+1}$ such that $$y_a=\int_{\mathbb{S}^n} X^\alpha  \mathrm{d}\mu(X)$$ for all $a\in I_{n+1,d}$. In this sense, $y$ is a \textit{homogeneous} truncated moment sequence and we denote the set of homogeneous truncated moment sequences admitting a representing measure supported on $\mathbb{S}^n$ by $\mathcal{R}_{n+1,d}^h$. By \cite[Corollary 2.4. and Theorem 3.1.]{FN2}, we know that for even degrees $$\overline{\mathcal{R}_{n,2d}}=\mathcal{R}_{n+1,2d}^h=\left\{y:=(y_a)_{a\in I_{n+1,d}}\subseteq\R\mid \mathcal{L}_y\in \mathcal{P}_{n+1,2d}^\vee\right\}\simeq \mathcal{P}_{n+1,2d}^\vee.$$ Thus, the dual cone $\mathcal{P}_{n+1,2d}^\vee$ can be interpreted as the cone of homogeneous truncated moment sequences that admit representing measures supported on $\mathbb{S}^n$. 

Similar to our primal setting in \Cref{Sec:ConeProp}, testing membership in $\mathcal{P}_{n+1,2d}^\vee$ is a challenging task. Therefore, this task is often replaced by membership tests to the larger well-understood dual cone $\Sigma_{n+1,2d}^\vee$. Yet, membership in $\Sigma_{n+1,2d}^\vee$ is only necessary but not sufficient for membership in the subcone $\mathcal{P}_{n+1,2d}^\vee$. Moreover, (\ref{int:separation}) yields
$$\begin{cases} \Sigma_{n+1,2d}^\vee=C_0^\vee=\ldots=C_n^\vee=C_{n+1}^\vee\supsetneq \ldots \supsetneq C_{k-n}^\vee=\mathcal{P}_{n+1,2d}^\vee, & \mbox{if } n=2 \\
	\Sigma_{n+1,2d}^\vee=C_0^\vee=\ldots=C_n^\vee\supsetneq C_{n+1}^\vee\supsetneq \ldots \supsetneq C_{k-n}^\vee=\mathcal{P}_{n+1,2d}^\vee, & \mbox{else.}
\end{cases}$$ Therefore, for a strictly separating intermediate cone $C_i$ of $\Sigma_{n+1,2d}$ and $\mathcal{P}_{n+1,2d}$, we propose to establish membership tests for the dual cone $C_i^\vee$. The benefit of such tests is that they provide an affirmative answer for fewer linear functionals $L\colon\mathcal{F}_{n+1,2d}\to\R$ than membership tests for $\Sigma_{n+1,2d}^\vee$. %Indeed, $\mathcal{P}_{n+1,2d}^\vee \subsetneq C_i^\vee\subsetneq \Sigma_{n+1,2d}^\vee$ implies that $C_i^\vee$ approximates $\mathcal{P}_{n+1,2d}^\vee$ better than $\Sigma_{n+1,2d}^\vee$. These consideration 
This bring us to the two problems below.

\begin{customqu}{2} Give an explicit description of $C_i^\vee$. %for each strictly separating intermediate cone $C_i$ between $\Sigma_{n+1,2d}$ and $\mathcal{P}_{n+1,2d}$ in (\ref{ConeFiltration}).
\end{customqu}

\begin{customqu}{3} Interpret $C_i^\vee$ in the setting of the homogeneous truncated moment problem. %for each strictly separating intermediate cone $C_i$ between $\Sigma_{n+1,2d}$ and $\mathcal{P}_{n+1,2d}$ in (\ref{ConeFiltration}).
\end{customqu}

We are not only interested in $(n+1)$-ary $2d$-ics but also in those forms that are $\R$-linear combinations of the monomials $m_sm_t$ for $0\leq s,t\leq n+i$ for an a priori fixed $i\in\{0,\ldots,k-n\}$. Indeed, in \cite[Theorem 3.1]{GHK2}, we illuminated a connection between the cone $C_i$ and the $\R$-vector space $\mathfrak{S}_{n+i}:=\mathrm{span}_{\R}\{m_sm_t \mid 0\leq s,t\leq n+i\}$ by showing 
\begin{equation} \label{eq:spec} C_i\cap \mathfrak{S}_{n+i} = \mathcal{P}_{n+1,2d}\cap \mathfrak{S}_{n+i}.
\end{equation} This brings us to the \textit{$\mathcal{A}$-truncated moment problem} as elaborated below.

\subsubsection*{\textbf{$\mathcal{A}$-truncated moment problem}} This problem was studied by Nie \cite{Nie} and finds application in sparse polynomial optimization (cf.\ \cite{Lasserre1}). %However, up to date, only few results are known.
It asks if for a given finite set $\mathcal{A}\subseteq \N^{n}_0$ and a given \textit{truncated multisequence} $(y_{\boldsymbol{a}})_{\boldsymbol{a}\in \mathcal{A}}\subseteq\R$, there exists a non-negative Borel measure $\mu$ such that $$y_{\boldsymbol{a}}=\int \boldsymbol{X}^{\boldsymbol{a}}\mathrm{d}\mu(\boldsymbol{X})$$ for all $\boldsymbol{a}\in\mathcal{A}$. In the light of (\ref{eq:spec}), we therefore pose the following problem.

\begin{customqu}{4}
	For $i=0,\ldots,k-n$, consider the finite set $$\mathcal{A}_{i}:=\{\alpha_s+\alpha_t \mid 0\leq s,t\leq n+i\}\subseteq I_{n+1,2d} \subseteq \N_0^{n+1}$$ and relate $C_i^\vee$ to the $\mathcal{A}_i$-truncated moment problem.
\end{customqu}

\subsection{Cones along Toric Varieties} \label{sec:toric}
For our investigations in \cite{GHK, GHK2} and this paper, we chose $I_{n+1,d}$ to be lexicographically ordered. This choice was crucial in many proofs and cannot be neglected. We now illustrate with the help of \Cref{exmp:terdec} what happens in the non-Hilbert case of ternary decics if $I_{n+1,d}$ is ordered by another monomial order. %other than the lexicographical.

\begin{exmp} \label{exmp:terdec} \textsc{Ternary Decics} \\
	Here, $n=2$, $d=5$ and $k=k(2,5)=20$. For $\alpha,\beta\in\N_0^3$, we set $\alpha < \beta$ if %$\alpha_0+\alpha_1 < \beta_0+\beta_1$, or $\alpha_0+\alpha_1=\beta_0+\beta_1$ and $\alpha_0<\beta_0$, or $(\alpha_0,\alpha_1)=(\beta_0,\beta_1)$ and $\alpha_2<\beta_2$.
	\begin{itemize}
		\item $\alpha_0+\alpha_1 < \beta_0+\beta_1$ or 
		\item $\alpha_0+\alpha_1=\beta_0+\beta_1$ and $\alpha_0<\beta_0$ or
		\item $(\alpha_0,\alpha_1)=(\beta_0,\beta_1)$ and $\alpha_2<\beta_2$.
	\end{itemize}
	A straight forward computation verifies that $\leq$ is a monomial order. Therefore, ordering $I_{3,5}$ by $\leq$ starting with the greatest element, we obtain the ordered monomial basis $\{m_0,\ldots,m_{20}\}$ and we see that $m_0(X)=X_0^5$, $m_1(X)=X_0^4X_1$, $m_2(X)=X_0^3X_1^2$, $m_3(X)=X_0^2X_1^3$, $m_4(X)=X_0X_1^4$, $m_5(X)= X_1^5$. We furthermore compute
	\begin{eqnarray*}
		T_0&:=&\left\{[z]\in\PP^{20} \mid \exists x\in \C^3 \colon z_0=x_0^5\right\}=\PP^{20}, \\
		T_1&:=&\left\{[z]\in\PP^{20} \mid \exists x\in \C^3 \colon (z_0,z_1)=\left(x_0^5,x_0^4x_1\right)\right\}= T_0, \\
		T_2&:=&\left\{[z]\in\PP^{20} \mid \exists x\in \C^3 \colon (z_0,z_1,z_2)=\left(x_0^5,x_0^4x_1,x_0^3x_1^2\right)\right\}\subsetneq T_1, \\
		T_3&:=&\left\{[z]\in\PP^{20} \mid \exists x\in \C^3 \colon (z_0,z_1,z_2,z_3)=\left(x_0^5,x_0^4x_1,x_0^3x_1^2,x_0^2x_1^3\right)\right\}\subsetneq T_2, \\
		T_4&:=&\left\{[z]\in\PP^{20} \mid \exists x\in \C^3 \colon (z_0,z_1,z_2,z_3,z_4)=\left(x_0^5,x_0^4x_1,x_0^3x_1^2,x_0^2x_1^3,x_0x_1^4\right)\right\}\subsetneq T_3, \\
		T_5&:=&\left\{[z]\in\PP^{20} \mid \exists x\in \C^3 \colon (z_0,\ldots,z_5)=\left(x_0^5,x_0^4x_1,x_0^3x_1^2,x_0^2x_1^3,x_0x_1^4,x_1^5\right)\right\}\subsetneq T_4, \\
		T_6&:=&\left\{[z]\in\PP^{20} \mid \exists x\in \C^3 \colon (z_0,\ldots,z_6)=\left(x_0^5,\ldots,x_1^5,x_0^4x_2\right)\right\}=T_5, \\
		T_7&:=&\left\{[z]\in\PP^{20} \mid \exists x\in \C^3 \colon (z_0,\ldots,z_7)=\left(x_0^5,\ldots,x_0^4x_2,x_0^3x_1x_2\right)\right\}\subsetneq T_6, \\
		&\vdots & \\
		T_{20}&:=& \left\{[z]\in\PP^{20} \mid \exists x\in \C^3 \colon (z_0,\ldots,z_{20})=\left(x_0^5,\ldots,x_1x_2^4,x_2^5\right)\right\}\subsetneq T_{19}.
	\end{eqnarray*}
	Thus, we set $H_0:=T_0$, $H_1:=T_2$, $H_2:=T_3$, $H_3:=T_4$, $H_4:=T_5$, $H_5:=T_7$, $H_i:=T_{i+2}$ for $i=6,\ldots,18$ and let $V_i$ be the Zariski closure of $H_i$ in $\PP^{20}$ for $i=0,\ldots,18$. This construction gives us a specific filtration of projective varieties
	$$V(\PP^2)=V_{18} \subsetneq \ldots \subsetneq V_0=\PP^{20}$$ with a corresponding specific filtration of sets of real points 
	$$V(\PP^2)(\R)=V_{18}(\R) \subsetneq \ldots \subsetneq V_0(\R)=\PP^{20}(\R)$$ such that each inclusion appearing is strict. For $i=0,\ldots,18$, we lastly set $C_i:=C_{V_i}$ and obtain a specific cone filtration 
	\begin{equation} \label{eq:critfilt} \Sigma_{3,10}=C_0\subseteq \ldots \subseteq C_{18}=\mathcal{P}_{3,10}.
	\end{equation} In this filtration at least one inclusion has to be strict by Hilbert's 1888 theorem but \textit{it is not clear how many of these inclusions are strict and which ones.}
	
	Let us examine the subfiltration 
	\begin{equation} \label{eq:subfilt} \Sigma_{3,10}=C_0\subseteq C_1 \subseteq C_2 \subseteq C_3 \subseteq C_4
	\end{equation} of the cone filtration (\ref{eq:critfilt}) for strict inclusions. To this end, we set
	\begin{align*}
		q_1(Z)&:=Z_0Z_2-Z_1^2 & \mbox{and} \hspace*{1.8cm} p_1(\boldsymbol{Z})&:=Z_2-Z_1^2, \\
		q_2(Z)&:=Z_0Z_3-Z_1Z_2 & \mbox{and}\hspace*{1.8cm} p_2(\boldsymbol{Z})&:=Z_3-Z_1Z_2, \\
		q_3(Z)&=Z_0Z_4-Z_1Z_3 & \mbox{and} \hspace*{1.8cm}p_3(\boldsymbol{Z})&=Z_4-Z_1Z_3, \\ 
		q_4(Z)&=Z_0Z_5-Z_1Z_4 & \mbox{and}\hspace*{1.8cm} p_4(\boldsymbol{Z})&=Z_5-Z_1Z_4
	\end{align*} and consider the affine varieties $K_0:=\C^{20}$, $K_i:=\boldsymbol{\mathcal{V}}(p_1,\ldots,p_i)\subseteq \C^{20}$ for $i=1,\ldots,4$. Moreover, we let $W_i\subseteq\PP^{20}$ be the projective closure of $K_i\subseteq\C^{20}$ for $i=0,\ldots,4$. Similarly as in \cite[Theorem 1]{GHK}, we observe that $V_i=W_i$ for $i=0,\ldots,4$ and thus $V_i$ is a non-degenerate irreducible totally-real projective variety of codimension $i$ for $i=0,\ldots,4$. Moreover, we see that $V_0,\ldots,V_4$ are cones over the varieties $\tilde{V}_0,\ldots,\tilde{V}_{4}$ that are parametrized by
	\begingroup
	\allowdisplaybreaks
	\begin{align*} 
		\chi_0 \colon \PP^2 & \dashrightarrow  \PP^{1} \\ \ [x] & \mapsto  \left[x_0^5:x_0^4x_1\right], \\
		\chi_1 \colon \PP^2 & \dashrightarrow  \PP^{2} \\ \ [x] & \mapsto  \left[x_0^5:x_0^4x_1:x_0^3x_1^2\right],  \\
		\chi_2 \colon \PP^2 & \dashrightarrow  \PP^{3} \\ \ [x] & \mapsto  \left[x_0^5:x_0^4x_1:x_0^3x_1^2:x_0^2x_1^3\right], \\
		\chi_3 \colon \PP^2 & \dashrightarrow  \PP^{4} \\ \ [x] & \mapsto  \left[x_0^5:x_0^4x_1:x_0^3x_1^2:x_0^2x_1^3:x_0x_1^4\right], \\ 
		\chi_4 \colon \PP^2 & \dashrightarrow  \PP^{5} \\ \ [x] & \mapsto  \left[x_0^5:x_0^4x_1:x_0^3x_1^2:x_0^2x_1^3:x_0x_1^4:x_1^5\right].
	\end{align*} 
	\endgroup
	In particular, $\chi_4$ can be interpreted as the Veronese embedding (of degree $5$) \begin{eqnarray*} V^{(5)} \colon \PP^1 & \to & \PP^5, \\ 
		\ [x_0:x_1] &\mapsto & \left[x_0^5:x_0^4x_1:x_0^3x_1^2:x_0^2x_1^3:x_0x_1^4:x_1^5\right] \end{eqnarray*}
	 and thus $V_4$ is a multiple cone over the Veronese variety $V^{(5)}(\PP^1)$ in $\PP^{k(1,5)}$. According to \cite[Example 13.4.]{Harris}, the Hilbert polynomial of the irreducible projective variety $V^{(5)}(\PP^1)$ is given by $$p_{V^{(5)}(\PP^1)}(T)=\dbinom{5T+1}{1}=5T+1\in\C[T].$$
	 Therefore, $\deg\left(V^{(5)}(\PP^1)\right)=5$ and thus, we conclude that $V_4$ has degree $5$. %by \Cref{thm:degreecones}. 
	
	Altogether, $V_4$ is a non-degenerate irreducible totally-real projective variety of minimal degree such that $V(\PP^n)(\R)\subseteq V_4(\R)$. Therefore, \cite[Proposition 1]{GHK} yields $\Sigma_{3,10}=C_4$, which shows that the subfiltration (\ref{eq:subfilt}) of the cone filtration (\ref{eq:critfilt}) collapses. Hence, there are at most $13$ strictly separating intermediate cones between $\Sigma_{3,10}$ and $\mathcal{P}_{3,10}$ in (\ref{eq:critfilt}). On contrary to this, with the lexicographic order, there were $\mu(2,5)=k(2,5)-4-2=14$ strictly separating intermediate cones between $\Sigma_{3,10}$ and $\mathcal{P}_{3,10}$ in (\ref{ConeFiltration}).
\end{exmp}

The \Cref{exmp:terdec} above illustrates that different monomial orders may lead to different behaviours of inclusions in corresponding cone filtrations %
(similar to (\ref{ConeFiltration}) and (\ref{eq:critfilt}))
that can be constructed by our method. Therefore, it would be interesting to investigate \textit{which properties of a monomial order influence the occurence of cone equalities and inequalities, and how?} \\

The specific projective varieties considered in \cite{GHK, GHK2} and this paper %, and more generally also the projective varieties along a monomial order in the light of our construction, 
are projective varieties that are described by monomials. This brings us to \textit{toric varieties} and we refer an interested reader to \cite{CLS} for an introduction to this concept. For the convenience of the reader, here, we include the definition of a projective toric variety in the framework of this paper.

\begin{defn} For a finite set $\mathcal{B}:=\{\beta_0,\ldots,\beta_s\}\subseteq I_{n+1,d}\subseteq \Z^{n+1}$ ($s\in\N$), we set
	\begin{eqnarray*} \phi_\mathcal{B} \colon (\C^\times)^{n+1} &\to& (\C^\times)^{s+1} \\
		x&\mapsto & \left(x^{\beta_0},\ldots,x^{\beta_s}\right).
	\end{eqnarray*} and consider the canonical map $\pi \colon (\C^\times)^{s+1} \to \PP^{s}, y\mapsto [y]$. The Zariski closure of the image of $\pi\circ\phi_{\mathcal{B}}$ in $\PP^s$ is a \textit{projective toric variety}.
\end{defn}

The projective varieties $\tilde{V}_0,\ldots,\tilde{V}_{k-n}$ are thus projective toric varieties and, in this sense, also $V_0,\ldots,V_{k-n}\subseteq\PP^k$ can be interpreted as projective toric varieties. The same remains true when $I_{n+1,d}$ is ordered by a monomial order other than the lexicographical. 

Since toric varieties are well-understood, they might serve as a good starting point for further investigations that attempt to establish a general criterion for an intermediate cone $C_W$ between $\Sigma_{n+1,2d}$ and $\mathcal{P}_{n+1,2d}$ along a projective variety $W$ to be strictly separating in non-Hilbert cases. We therefore propose the following question.

\begin{customqu}{5} \label{OP:1}
	For projective toric varieties $\mathfrak{W}_1$, $\mathfrak{W}_2\subseteq\PP^k$ such that $\mathfrak{W}_1\subseteq \mathfrak{W}_2$, let $q\in\mathcal{F}_{k+1,2}$ be locally PSD on $\mathfrak{W}_1(\R)$. When does there exist a quadratic form $\mathfrak{q}\in\mathcal{F}_{k+1,2}$ such that $\mathfrak{q}$ vanishes on $V(\PP^n)$ and $(q+\mathfrak{q})\vert_{_{\mathfrak{W}_2(\R)}}\geq 0$?
\end{customqu}

Replacing the projective toric variety $V(\PP^n)$ by some arbitrary projective toric variety $\mathfrak{W}_0\subseteq\PP^k$ in \Cref{OP:1}, we obtain the more general problem below. %whose answer will be a first step towards a complete answer for \cite[]{}. 

\begin{customqu}{6}
	For three projective toric varieties $\mathfrak{W}_0$, $\mathfrak{W}_1$, $\mathfrak{W}_2\subseteq\PP^k$ such that $\mathfrak{W}_0\subseteq \mathfrak{W}_1\subseteq \mathfrak{W}_2$, let $q\in\mathcal{F}_{k+1,2}$ be locally PSD on $\mathfrak{W}_1(\R)$. When does there exist a quadratic form $\mathfrak{q}\in\mathcal{F}_{k+1,2}$ such that $\mathfrak{q}$ vanishes on $\mathfrak{W}_0$ and $(q+\mathfrak{q})\vert_{_{\mathfrak{W}_2(\R)}}\geq 0$?
\end{customqu}

Our considerations regarding membership tests for the intermediate cones $C_i$'s and the duals of $C_i$'s as in \Cref{Sec:ConeProp} and \Cref{Sec:Dual} can be reconsidered for intermediate cones coming from arbitrary projective toric vareities. 

%\subsection{Computational Value and Optimizational Perspective}
%\section{Concluding Remarks} \textcolor{red}{if required}
\section*{Acknowledgements} 
  The authors would like to thank \textit{Mathematisches Forschungsinstitut Oberwolfach} for its hospitality in March 2023. The first author acknowledges support through \textit{Oberwolfach Leibniz Fellows} program. The second author is grateful for the support provided by the scholarship program of the University of Konstanz under the \textit{Landesgraduiertenfördergesetz} and the \textit{Studienstiftung des deutschen Volkes}. The third author acknowledges the support of \textit{Ausschuss für Forschungsfragen der Universität Konstanz}.
%%%%%%%%%%%%%%%%%%%%%%%%%%%%%%%%%%%%%%%%%%%%%%%%%%%%%%%%%%%%%%%%%%%%%%%%%%%%%%%%%%%%%%%%%%%%%%%%%%%%%%%%%%%%%%%%%%%%%%%%%%%%%%%%

\end{document}